\documentclass[submission]{dmtcs-episciences}
\usepackage[utf8]{inputenc}

\usepackage{color}
\usepackage{cite}
\usepackage{enumerate}
\usepackage{amssymb,amsmath}
\usepackage{algorithm}
\usepackage{algpseudocode}
\algrenewcommand{\algorithmicrequire}{\textbf{Input:}}
\algrenewcommand{\algorithmicensure}{\textbf{Output:}}

\newcommand{\tn}{\textnormal}

\newtheorem{definition}{Definition}
\newtheorem{theorem}{Theorem}
\newtheorem{lemma}[theorem]{Lemma}
\newtheorem{proposition}[theorem]{Proposition}
\newtheorem{corollary}[theorem]{Corollary}

\begin{document}
%


\author{Amit Kumar Srivastava\affiliationmark{1}}
\title[Formatting an article for DMTCS]{Ins-Robust Primitive Words}

\affiliation{
  Department of CSE, IIT Guwahati, India}



\date{Received: date / Accepted: date}
\maketitle

\begin{abstract}
Let $Q$ be the set of primitive words over a finite alphabet with at least two symbols. We characterize a class of primitive words, $Q_I$, referred to as ins-robust primitive words, which remain primitive on insertion of any letter from the alphabet and present some properties that characterizes words in the set $Q_I$. It is shown that the language $Q_I$ is dense. We prove that the language of primitive words that are not ins-robust is not context-free. We also present a linear time
algorithm to recognize ins-robust primitive words and give a lower bound on the number of $n$-length ins-robust primitive words.
\end{abstract}


\section{Introduction}\label{sec:bterm}
A word is a sequence of finite symbols or letters from a finite alphabet. Combinatorics on words is the study of mathematical and computational problems related to words. A word is said to be primitive if it is not a proper power of a shorter word \cite{lothaire1997combinatorics}. The notion of primitive words plays a central role in the field of combinatorics on words and algebraic coding theory.
Apart from this, primitive words have also received much attention in the area of formal langauges.
This paper presents a characterization of the set of a special type of primitive words which remain robust under an operation of insertion of a single letter from the underlying alphabet.

The relation between the language of primitive words and other formal languages has been
thoroughly explored \cite{lischke2011primitive,domosi1993formal,petersen1994ambiguity}.
In 1991, D\"{o}m\"{o}si et al. raised the problem whether the language of primitive words, $Q$, is context-free? This long standing open problem has been one of the main investigations about primitive words \cite{doi}. A linear time algorithm to test whether a given word is primitive is given in \cite{duval2004linear}.


The robustness of the language of primitive words is concerned with the preservation of primitivity of a primitive word with respect to the various point mutation
operations such as insertion, deletion or substitution of a symbol and homomorphism \cite{puaun2002robustness, srivastava2016robust}. A primitive word, $u$, is said to be ins-robust if a
word remains primitive on insertion of any letter in $u$ \cite{puaun2002robustness}. In this paper, we investigate the
language of ins-robust primitive words. In particular, our contributions are as
follows.
\begin{enumerate}[(a)]
\item We characterize ins-robust primitive words and identify several properties. 
\item We show that the set of ins-robust primitive words is dense. 
\item We prove that the language of ins-robust primitive words is not regular. 
\item We show that the language of primitive words that are not ins-robust, is not context-free. 
\item We give a lower bound on the number of ins-robust primitive words of a given
length.
\item We give a linear time algorithm to test if a word is ins-robust primitive.
\end{enumerate}

The paper is organized as follows. The next section reviews the
basic concepts on words and some existing results on robustness of primitive words.
In Section \ref{sec:ins} we characterize ins-robust primitive words and study
several properties of the language of ins-robust primitive words. In Section \ref{sec:den} we show that the set of ins-robust primitive words is dense. In Section \ref{sec:for} we give proof for the claim that language of ins-robust primitive words is not regular and the language of primitive words which is not ins-robust is not context-free.
In Section \ref{sec:cou} we give a lower bound on the number of ins-robust primitive words of
a given length.
In Section \ref{sec:alg} we give a linear time algorithm to recognize a
ins-robust primitive word.
Finally, conclusions and some open problems are presented in
Section \ref{sec:con}.



\section{Preliminaries}\label{sec:pre}
Let $V$ be a finite alphabet with at least two distinct elements. We call the elements of $V$ as letters or symbols. The set $V^*$ is the collection of all words (strings) generated by the letters from the alphabet $V$. The $concatenation$ of two words $u$ and $v$ is represented as $u.v$ or simply $uv$. The length of a word $w$ is denoted as $|w|$.
We use notation $|u|_a$  for the number of occurrences of symbol $a \in V$ in a word $u$.
The set $V^n$ is the set of all words having length $n$ over the alphabet $V$. We define $V^{*} = \bigcup_{n \in N} V^n$ and $V^{+} = V^{*}\setminus \{\lambda\}$, where $V^{0} = \{\lambda\}$ and $\lambda$ is the empty string, that is, $|\lambda| = 0$.
The set $V^*$ is the free monoid and $V^+ = V^* \setminus \{\lambda\}$ is the free semigroup under the concatenation operation.

Any set $L \subseteq V^{*}$ is a language over the alphabet $V$. A language $L$ over $V$ is
said to {\it reflective} if $uv \in L$ implies $vu \in L$ where $u,v \in V^*$. The word $u$ ($v$) is called {\it prefix} ({\it suffix}) of a word $w$ if $w$ can be written as $uv$. A prefix (suffix) with length $k$ of a word $u$ is denoted by $\tn{\it pref}(u,k)$
($\tn{suff}(u,k)$, respectively) where $k \in \{0, 1, \ldots, |u|\}$. For $k=0$, $\tn{\it pref}(u,0) = \tn{\it suff}(u, 0) = \lambda$. We use notation $\tn{\it pref}(u)$ to specify the set of all non-empty prefixes
of a word $u$. If $w$ can be written as $xyz$, where $y \in V^{+}$ and $x,z \in V^{*}$, then the word $y$ is called as a \textit{factor} of the word $w$. Let $w=a_1 \ldots a_n$ be a word, where $a_i \in V$ for $i \in
\{1, \ldots, n\}$.
 The reverse of the word $w$ is $rev(w) = a_n a_{n-1} \ldots a_2 a_1$.  The factor $a_i a_{i+1} \ldots a_j$ of $w$ is denoted by $w[i..j]$, where $i \leq j$. The cardinality of a set $X$ is denoted by $|X|$. The language of nonprimitive words is denoted as $Z$.
For  elementary notions and results in formal languages theory, we refer to \cite{hopcroft1979introduction}.

A well studied theme on words is properties about periodicity and primitivity \cite{lischke2011primitive,shyr1994non,petersen1994ambiguity,domosi1993formal,puaun2002robustness}
and associated counting and sampling problems \cite{lijun2001count}. A word $w = a_1 a_2 \ldots a_n$ is said to be periodic if $a_i = a_{i+p}$ for some $p > 0$ and $1 \le i \le n-p$. Such a number $p$ is called period of the word $w$. The ratio $ e = \frac{|w|}{ p} $ is called as the exponent of the word $w$. A word $w$ is said to be a $repetition$ if and only if $e \ge 2$. A $maximal \ repetition$ in a word is a factor which is a repetition such that its extension by one letter to the right or to the left yields a word with a larger period.  Equivalently, a factor $r = w[i..j]$ is a maximal repetition of a word $w$ if no factor of $w$ that contains $r$ as a proper factor has the same period as $r$ \cite{kolpakov1999finding,kolpakov1999maximal}. For example, the factor $ababa$ in the word $w=abaababaabaab$ is a maximal
repetition with period 2, while the factor $abab$ is not a maximal repetition.

%
 A word, $w$, is said to be primitive if $w$ cannot be expressed as a non-trivial power of another word. In other words, if $w = v^n$ for some word $v$ and $w$ is primitive then $n=1$ and $w=v$. As the definition is meaningful only when the underlying alphabet has at least two letters, we assume throughout that $V$ is a non-trivial alphabet with at least two distinct symbols. Primitive words have been extensively studied in the literature, see for example \cite{lischke2011primitive,puaun2002robustness,shyr1994non,petersen1994ambiguity,domosi1993formal}.

\section{Characterization of Ins-Robust Primitive Words}\label{sec:ins}
As mentioned in the previous section the set of all words over an alphabet can be partitioned in two sets, those which can not be written as power of any other smaller word, known as primitive words and the other remaining are non-primitive words. We can further partition the set of primitive words in different classes. In this section we study a class of primitive words that remains primitive on insertion of any arbitrary symbol from the alphabet. This special class is known as ins-robust primitive words. A formal definition is as follows.

\begin{definition}[Ins-Robust Primitive Word]
A primitive word $w$ of length $n$ is said to be ins-robust primitive word if the word
$$\tn{pref}(w, i)\ . \ a \ .\ \tn{suff}(w, n - i)$$ is a primitive word for all $i \in \{0, 1, \ldots, n \}$ where $a\in V$.
\end{definition}

There are infinitely many primitive words which are ins-robust. For example, the words $abba$ and $a^m b^n$ for $m,n \ge 2$ are ins-robust primitive words.
We denote the set of all ins-robust primitive words over an alphabet $V$
by $Q_I$. Clearly the language of ins-robust primitive words is a subset of the set
of primitive words, that is, $ Q_I \subset Q$.
The following results ensure that if a given word $w$ is not primitive then a word obtained after inserting a symbol from $V$ in $w$ will be primitive.
\begin{lemma}[\!\cite{puaun2002robustness}]\label{lem:ext}
If $u \in V^+$, $u \neq a^n$ for any $a \in V$, $n \geq 1$ then at least one
word among the words $u$ and $ua$ is primitive.
\end{lemma}
\begin{lemma}[\!\cite{puaun2002robustness}]\label{cor:prins}
If $u_1, u_2 \in V^+$, $u_1 u_2 \neq a^n$, for any $a \in V$, $n \geq 1$ then at least one
of the words among $u_1 u_2$, $u_1 a u_2$ is primitive.
\end{lemma}

The above results help us to completely characterize the primitive words which are not ins-robust.
Next we give a structural characterization of the words
that are in the set $Q$ but not in the set $Q_I$.

\begin{theorem} \label{thm:inschar}
A primitive word $w$ is not ins-robust if and only if $w$ can be expressed in the form of
$u^{r} u_1 u_2 u^{s}$ where $u = u_1 c u_2\in Q$, $u_1$, $u_2 \in V^*$, for some $c \in
V$, $r$, $s \ge 0$ and $r + s \ge 1$.
\end{theorem}

\begin{proof}
We prove the sufficient and necessary conditions below.
\begin{itemize}
\item[$(\Leftarrow)$] This part is straightforward. Let us consider a word $w = u^{k_1} u_1 u_2 u^{k_2}$ where $u_1 c u_2 = u$ for some $c \in V$.
The word $w$ is primitive by Lemma \ref{cor:prins}.
Now insertion of the letter $c$ in $w$ (between $u_1$ and $u_2$) gives the exact power of $u$ which become a non-primitive word. Hence, $w$ is not an ins-robust primitive word.\\
\item[$(\Rightarrow)$] Let $w$ be a primitive word but not ins-robust. Then there exists a
decomposition $w=w_1  w_2$ such that $w_1 c w_2$ is not a primitive word for some letter $c \in V$.
That is, $w_1 c w_2 = p^n $ for some $p \in Q$ and $n \ge 2$.
Therefore  $w_1  = p^r p_1 $ and  $ w_2 = p_2 p^s  $ for $r,s \ge 0 \ and \ r+s \ge 1$ such that $ p_1 c p_2 = p$.
Hence $w= p^r p_1  p_2 p^s$.
\end{itemize} 
\end{proof}

\begin{definition}[Non-Ins-Robust Primitive Words]
A primitive word $w$ is said to be non-ins-robust if $w \in Q$ but $w \notin Q_I$. We denote the set of all non-ins-robust primitive words as $Q_{\overline{I}}$. So, $Q_{\overline{I}} = Q \setminus Q_I$, where `$\setminus$' is the set difference operator.
\end{definition}

The next theorem is about an equation in words and identifies a sufficient condition under which three words are power of a common word.
\begin{theorem}[\cite{lischke2011primitive}]\label{thm:equ}
If $u^m v^n = w^k \ne \lambda$ for words $u$, $v$, $w \in V^{*}$ and natural
numbers $m$, $n$, $k \ge 2$, then $u$, $v$ and $w$ are powers of a common word.
\end{theorem}

The following lemma is a consequence of the Theorem \ref{thm:equ} which states
that a word obtained by concatenating powers of two distinct primitive
words is also primitive.
\begin{lemma}[\cite{lischke2011primitive}] \label{lem:piqj}
If $p,q \in Q$ with $p \neq q$ then $p^i q^j \in Q$ for all $i$, $j \ge 2$.
\end{lemma}

\begin{proposition}
If $u ,v \in Q$, $u^m = u_1  u_2 $ and $v = u_1 c u_2$ for some $c \in V$ then $u^m v^n \in Q_{\overline{I}}$ for $m, n \ge 2$.
\end{proposition}
\begin{proof}
From Lemma \ref{cor:prins} we know that at least one of $u_1 u_2$ and $u_1 c u_2$ is primitive. Since $u^m = u_1  u_2$ for $m \ge 2$ and $v=u_1 c u_2$, therefore $v$ is primitive and so is $u^m v^n = u_1  u_2(u_1 c u_2)^n$. After insertion of the letter $c$ we will get $(u_1 c u_2)^{n+1}$
which is not a primitive word. However, by Lemma \ref{lem:piqj}, $u^m v^n$ is a primitive word for $m,n \ge 2$. Hence it is not a ins-robust word, that is, $u^m v^n \in Q_{\overline{I}}$. 
\end{proof}

An existing result shows that if a word $w$ is primitive then $rev(w)$ is also primitive. We prove the similar notion for ins-robust primitive word.
\begin{lemma}
If $w \in Q_I$ then $rev(w) \in Q_I$.
\end{lemma}
\begin{proof}[By contradiction] Assume that for a word $w \in Q_I$, $rev(w)$ is not a ins-robust primitive word. i.e. $rev(w) = p^r p_1  p_2 p^s$ where $p = p_1 c p_2 \in Q$ for some $c \in V$.
Then the word $w = rev(rev(w)) =$ $rev(p^r p_1  p_2 p^s)$ =   $(rev(p))^s \ rev(p_2)$ $ \ rev(p_1) \ (rev(p))^r$ and $p= p_1 c p_2$, $rev(p) = rev(p_2) ~ c ~ rev(p_1)$. By Lemma \ref{cor:prins}, $w$ is not a ins-robust primitive word, which is a contradiction.
Therefore, if $w \in Q_I$ then $rev(w) \in Q_I$.
\end{proof}

We know that the language of primitive words $Q$ and the language of
non-primitive words $Z$ over an alphabet $V$ are reflective. Similarly, we have the property of reflectivity for the language of ins-robust primitive words $Q_I$.

\begin{lemma}[\! \cite{puaun2002robustness}] \label{lem:refl}
The languages $Q$ and $Z$ are reflective.
\end{lemma}

\begin{theorem} \label{thm:insref}
$Q_I$ is reflective.
\end{theorem}
\begin{proof}[By contradiction]
Let there be a word $w=xy \in Q_I$ such that $yx \notin Q_I$. Since $w \in Q_I$, hence $w \in Q$. By Lemma \ref{lem:refl}, we know that $Q$ is reflective. Therefore $yx \in Q$ and so $yx \in Q\setminus Q_I$, i.e. $yx \in Q_{\overline{I}}$. Using Theorem \ref{thm:inschar}, we have $yx = u^ru_1  u_2 u^s$ where $u =u_1 c u_2
\in V^*$ for some $c \in V$ and $r +s \ge 1$.
There can be three possibilities as follows.

\begin{description}
\item[\textbf{Case A}] If $y = u^{r_1} u'$, $x = u'' u^{r_2} u_1 u_2 u^s$ where $u = u' u''$ and $r_1 + r_2 + 1 = r$.

In this case $xy = u'' u^{r_2} u_1 u_2 u^s u^{r_1} u'$
= $(u'' u')^{r_2} u'' u_1 u_2 u' (u'' u')^{s + r_1}$.

Since $u = u_1 c u_2$, therefore $u'' u_1 c u_2 u' = u'' u u' = (u'' u')^2$.

Therefore $ (u'' u')^{r_2} u'' u_1 c u_2 u' (u'' u')^{s + r_1}$
= $(u'' u')^{s+r+1}$, 
i.e. $xy \in Q_{\overline{I}}$, which is a contradiction.

\item[\textbf{Case B}] $y = u^{r} u'$, $x = u'' u^s$ where $u' u'' = u_1 u_2$
\begin{description}
\item[\textbf{Case B.1}] If $u' = u_1'$ and $u'' = u_1'' u_2 $ where $u_1' u_1'' = u_1$.

Since $u = u_1 c u_2 = u_1' u_1'' c u_2$.

In this case $xy =  u'' u^s u^{r} u' = u_1'' u_2 u^s u^r u_1'$.

Now $u_1'' c u_2 u^s u^r u_1' = (u_1'' c u_2 u_1')^{r+s+1}$.

Therefore $xy \in Q_{\overline{I}}$, a contradiction.

\item[\textbf{Case B.2}] If $u' = u_1 u_2'$ and $u'' = u_2'' $ where $u_2' u_2'' = u_2$. This is similar to Case B.1.
\end{description}
\item[\textbf{Case C}] If $y = u^{r} u_1 u_2 u^{s_1} u'$, $x = u'' u^{s_2}$ where $u = u' u''$. This case is similar to the Case A.
\end{description}
Hence $Q_I$ is reflective. 
\end{proof}

\begin{corollary}
$Q_{\overline{I}}$ is reflective.
\end{corollary}
\begin{proof}
We prove it by contradiction. Let there be a word $w=xy \in Q_{\overline{I}}$ such that $yx \notin Q_{\overline{I}}$.
We have $xy \in Q$ and $Q$ is reflective, so $yx \in Q$.
 Therefore $yx \in Q \setminus  Q_{\overline{I}}$, i.e. $yx \in Q_I$. But $Q_I$ is reflective by Theorem \ref{thm:insref}, we have
$xy \in Q_I$, which is a contradiction. Hence $yx \in Q_{\overline{I}}$. 
\end{proof}

\begin{theorem}\label{thm:cycl}
A word $w$ is in the set $Q_{\overline{I}}$ if and only if it is of the form $u^n u'$ or its cyclic permutation
for some $u \in Q$, $u = u' a$ , $a \in V$ and $n \ge 1$.
\end{theorem}
\begin{proof} We prove the sufficient and necessary conditions below.
\begin{description}
\item[($\Rightarrow$)] Let $w \in Q_{\overline{I}}$, then $w$ can be written as $w = u^r u_1  u_2 u^s$ for some $u \ (=u_1 a u_2) \in Q$ and $a \in V$.
Since $Q_{\overline{I}}$ is reflective, therefore $u_2 u^s u^r u_1  = (u_2
u_1 a)^{r+s} u_2 u_1$ is also in $Q_{\overline{I}}$.

\item[($\Leftarrow$)] If a word $w$ is a cyclic permutation of $u^n u'$ for $n
\ge 1$ and $u = u' a$ then after insertion of a symbol $a$, it gives a cyclic permutation of $u^{n+1}$ which
is non-primitive (since $Z$ is reflective). Therefore, $w \in
Q_{\overline{I}}$.
\end{description} 
\end{proof}

We observe that a word $w$ is periodic with period $p$ ($\ge 2$) divides $|w| + 1$ and $p \le |w| $ then $w$ is non-ins-robust primitive word.
Since $Q_I$ is reflective, therefore any cyclic permutation of $w$ is also non-ins-robust primitive word.
We know that cyclic permutation of a primitive word is also primitive. Cyclic permutation of an ins-robust primitive word is primitive. In next result, we show that it remains ins-robust too.

\begin{corollary}
Cyclic permutation of a ins-robust primitive word is ins-robust.
\end{corollary}
\begin{proof}
Let $w \in Q_I$. Then cyclic permutation of $w$ will be $yx$ for some partition $w= xy$. Since $Q_I$ is reflective. Therefore $yx$ is also ins-robust primitive word.
This proves that any cyclic permutation of an ins-robust primitive word is ins-robust.
\end{proof}
Observe that not all infinite subsets of $Q$ are reflective. For example,
the subset $\{a^n b^n \mid n \geq 1\}$ of $Q$ over the alphabet $\{a, b\}$
is not reflective.

\section{Ins-robust Primitive Words and Density} \label{sec:den}
A language $L \subseteq V^*$ is called a \emph{dense} language if for every word $w \in V^*$, there exist words $x$, $y \in V^{*}$ such that $xwy \in L$.
A language $ L \subseteq V^*$ is called \emph{right $k$-dense} if for every $u \in V^*$ there exists a word $x \in V^*$ where $|x| \le k$ such that $ux \in L$. The language is said to be right dense if its right $k$-dense for every $k \ge 1$. 

It is easy to see that $Q$ is right dense (see Lemma \ref{lem:ext}). The following result shows the relation between the language of primitive words $Q$ with the density.

\begin{theorem} \label{den1}
If $|w| = n$ and $wa^n \in Q_{\overline{I}}$ where $w \notin a^*$, then  $wa^n = u^2 u_1 u_2$, $u=u_1bu_2$ for $b \ne a$.
\end{theorem}

\begin{proof}
Let $wa^n \in Q_{\overline{I}}$. Then $wa^n = u^r u_1 u_2 u^s$, where $u = u_1 b u_2 $, for some $b \in V$.
We claim that $r=2 \ and \ s=0$.
If $s \ge 1$ Then $|u| \le n+1$.

\textbf{Case 1.} In this case we can have two cases depending on the length of $u$.

\textbf{Case 1a.} $|u| = n+1$ if $wa^n = u_1 u_2 u$.
Hence $u = ba^n$ for some $b \ne a$. Therefore $|u| = n+1$ and so $|u_1 u_2| = n $.
$|wa^n| = 2n+1$, which is a contradiction as $|wa^n|=2n$.

\textbf{Case 1b.} If $s \ge 1$ and $|u| \le n$, then $u = a^r$, where $r = |u|$, and therefore $u_1u_2 = a^{r-1}$.
$wa^n$ = $a^{2n}$. This case also leads to a contradiction.

Therefore $s=0$.
Hence $wa^n = u^r u_1 u_2$.

\textbf{Case 2.} $r=1$. This case is not possible.
Because in this case $|wa^n|= |u u_1u_2|=2n$, which implies $|u|= (2n+1)/2$ a non-integral value.

\textbf{Case 3.} If $r \ge 2$. Then $|wa^n| = |u^r u_1 u_2|  = ((r+1)|u|-1)$.
Since $|u_1u_2|$ = $|u|-1$, therefore $u_1u_2 = a^k$, where $k = |u| -1$.

If $r \ge 3$ then $|wa^n| = |u^r u_1 u_2|  = ((r+1)|u|-1)$.
In this case $|u u_1u_2|= {\frac{2n+1}{r+1}}*2 -1$ $\le n$. Therefore $u = a^{k+1}$.
Hence $r \ge 3$ is also not possible.

Thus the only possibility is $r=2$, $|wa^n| = u^2 u_1u_2$.
Since $u_1u_2 = a^k$ and $wa^n \in Q$ therefore $w, u \notin a^* $, and so
$u = a^{k_1}ba^{k_2}$ where $k_1 + k_2 =k$ and $b \ne a$.

\end{proof}

\begin{lemma} \label{denseins1}
Let $V$ be an alphabet, $w \in V^*$ , $|w|=n$ and $a \in V$. If $wa^n \in Q_{\overline{I}}$ then for $b \ne a$, $wb^n \in Q_I$.
\end{lemma}

\begin{proof} Let $wa^n \in Q_{\overline{I}}$. Then, by Theorem \ref{den1} we have, \\
$
\begin{array}{ll}
wa^n = u^2 u_1u_2 \tn{ and } u_1u_2 = a^k.\\
u=a^{k_1}ba^{k_2} \tn{ where } a \ne b.\\
\end{array}
$\\[5pt]
Let $wc^n$ be also in $ Q_{\overline{I}}$ for some $c \ne a$. Then, by Theorem \ref{den1} we have,\\
$
\begin{array}{ll}
wc^n = v^2 v_1v_2 \tn{ and } v_1v_2 = c^k.\\
v=c^{k_1'}dc^{k_2'} \tn{ where } c \ne d.
\end{array}
$\\
But since $|u| = |v|$ and $w = uu' = vv'$ where $u' = u u_1 u_2$ and $v' = v v_1 v_2$, 
therefore $u=v$, that is,
$a^{k_1}ba^{k_2} = c^{k_1'}dc^{k_2'}$.\\
If $k_1 < k_1'$ then $a=b=c=d$, which is a contradiction. Alternatively, if $k_1 = k_1'$ then $a=c$ which is again a contradiction. Therefore $wc^n \in Q_I$.

\end{proof}

\begin{theorem}
The language $Q_I$ is dense over the alphabet $V$.
\end{theorem}
\begin{proof} Consider a word $w$. We only need to consider the case when $w \notin Q_I$, that is,  $w \in V^* \setminus Q_I$. By Lemma \ref{denseins1}, there exists $b \in V$ such that $wb^n \in Q_I$, where $n = |w|$.

\end{proof}

\section{Relation of $Q_I$ with the Other Formal Languages}\label{sec:for}
We now investigate the relation between the language of
ins-robust primitive words with the traditional languages in Chomsky hierarchy. We prove that the language of ins-robust primitive words over an alphabet is not regular and also show that the language of non-ins-robust primitive words is not context-free.
For completeness, we recall the pumping lemma for regular languages and pumpimg lemma for context-free languages which
will be used to show that $Q_{I}$ is not regular and
$Q_{\overline{I}}$ is not context-free respectively.

\begin{lemma}[Pumping Lemma for Regular Languages \cite{hopcroft1979introduction}]
 For a regular language $L$, there exists an integer $n > 0$ such that for every word $w \in L$ with $|w| \ge n$, there
exist a decomposition of $w$ as $w = xyz$ such that the following conditions
holds.
\begin{enumerate}[(i)]
\item $|y| > 0$,
\item $|xy| \le n$, and
\item $xy^i z \in L $ for all $i \ge 0$.
\end{enumerate}
\end{lemma}

Let us recall a result which will be used in proving that the language of ins-robust primitive words is not regular.
\begin{lemma}[\label{lem:sola}\cite{horvath2005language}]
For any fixed integer $k$, there exist a positive integer $m$ such that the
equation system $(k-j)x_j+j=m$, $j=0,1,2, \ldots, k-1$ has a nontrivial solution
with appropriate positive integers $x_1,x_2,\ldots,x_j > 1$.
\end{lemma}

\begin{theorem}
$Q_{I}$ is not regular.
\end{theorem}
\begin{proof}[By contradiction]
Suppose that the language of ins-primitive words $Q_{I}$ is regular. Then
there exist a natural number $n >0$ depending upon the number of states of finite automaton for
$Q_{I}$.

Consider the word $w=a^n b a^m b,m > n+1$ and $m \neq 2n$. Note that
$w$ is an ins-primitive word over $V$, where $|V| \ge 2$ and $a \ne b$. Since $w \in Q_{I}$
and $|w| \ge n$, then it must satisfy the other conditions of pumping Lemma
for regular languages. So there exist a decomposition of $w$ into
$x$, $y$ and $z$ such that $w= xyz,|y| >0$ and $x y^i z \in Q_{I}$ for all $i \ge 0$.

Let $x = a^k,\ y=a^{(n-j)},\ z=a^{j-k} b a^m b$. Now choose $i=x_j$ and
since we know by Lemma \ref{lem:sola} that for every $j
\in \{0,1, \ldots , n-1\}$, there exists a positive integer $x_j >1$ such that $x y^{x_j} z= a^k
a^{(n-j)x_j} a^{j-k} b a^m b = a^{(n-j)x_j+j} b a^m b =  a^m b a^m b = (a^m b)^2 \notin Q_{I}$ which is a contradiction. Hence the language of ins-primitive words $Q_{I}$ is not regular.

\end{proof}

\begin{lemma}[Pumping Lemma for Context-Free Languages \cite{hopcroft1979introduction}]
\label{lem:pcfl}
Let $L$ be a CFL. Then there exists an integer $n >0$ such that for every $u \in
L$ with $|u| \ge n$, $u$ can be decomposed into $vwxyz$ such that the following
conditions hold:
\begin{enumerate}[(a)]
 \item $|wxy| \le n$.
 \item $|wy| >0$.
 \item $v w^{i} x y^{i} z \in L$ for all $i \ge0$.
\end{enumerate}
\end{lemma}

\begin{theorem}\label{thm:cfl}
$Q_{\overline{I}}$ is not a context-free language for a binary alphabet.
\end{theorem}

\begin{proof}
Let $V= \{a, \ b \}$ be an alphabet. On contradiction, assume that $Q_{\overline{I}}$ is a context-free language. Let $p >0$ be
an integer which is the pumping length for the language $Q_{\overline{I}}$.
Consider the string $s = a^{p+1} b^{p+1} a^{p+1} b^p$, where $a , \ b \in V$ are
distinct. It is easy to see that $s\in Q_{\overline{I}}$ and $|s| \ge p$.

Hence, by the Pumping Lemma \ref{lem:pcfl}, $s$ can be written in the form $s = uvwxy$,
where $u, v, w, x,$ and $y$ are factors, such that $|vwx| \le p$, $|vx| \ge
1$, and $uv^iwx^iy$ is in $Q_{\overline{I}}$ for every integer $i \ge 0$. By the
choice of $s$ and the fact that $|vwx| \le p$, we have
one of the following possibilities for $vwx$:
\begin{enumerate}[(a)]
\item \label{e:a} $vwx = a^j$ for some $j \le p$.
\item \label{e:b} $vwx = a^jb^k$ for some $j$ and $k$ with $j+k \le p$.
\item \label{e:c} $vwx = b^j$ for some $j \le p$.
\item \label{e:d} $vwx = b^j a^k$ for some $j$ and $k$ with $j+k \le p$.
\end{enumerate}
In Case (\ref{e:a}), since $vwx = a^j$, therefore  $vx = a^ t$ for some $t \ge 1$ and hence $uv^iwx^iy = a^{p-t+1} b^{p+1} a^{p+1} b^p  \notin Q_{\overline{I}}$ for $i =0$.

Case (\ref{e:b}) can have several subcases.
\begin{enumerate}
\item  $v= a^{j_1}, \ w = a^{j_2}, \ x= a^{j_3}b^k$.
\item  $v= a^{j_1}, \ w = a^{j_2}b^{k_1}, \ x= b^{k_2}$.
\item  $v= a^{j}b^{k_1}, \ w = b^{k_2}, \ x= b^{k_3}$.
\end{enumerate}
In Case (1), Case (2) and Case (3) if we take $i = 4$, $uv^iwx^iy \ne Q_{\overline{I}}$.

Similarly, we can obtain contradiction in Case (\ref{e:c}) and Case (\ref{e:d}) by choosing a suitable $i$.
Therefore, our initial assumption that $Q_{\overline{I}}$  is context-free, must be false.
\end{proof}
Next we prove that the language of non-ins-robust primitive words is not context-free in general.
\begin{lemma}
The language $Q_{\overline{I}}$ is not context-free over an alphabet $V $ where $V$ has at least two distinct letters.
\end{lemma}
\begin{proof}
The proof of Theorem \ref{thm:cfl} can be generalized to arbitrary alphabet $V$ having
at least two letters. The set of all words over alphabet having greater than two distinct letters also contains the words with two letters. If $Q
_{\overline{I}}$ is assumed to be a CFL over $V $ where $|V| \ge 3$, then we can choose words of the form used in Theorem \ref{thm:cfl} and obtain a contradiction. Hence the language of non-ins-robust primitive words $Q_{ \overline{I}}$ is not context-free over $V$ where $|V| \ge 2$.

\end{proof}
\section{Counting Ins-Robust Primitive Words}\label{sec:cou}
In this section we give a lower bound on number of $n$-length ins-robust primitive words.
Let $V$ be an alphabet and $Z(k) = V^k \setminus Q$ be the set of $n$-length
non-primitive words.

We have the following result that gives the number of the primitive words of
length $m$.
\begin{proposition} [\cite{lijun2001count}] \label{count}
Let $m \in N$ and $m= {m_1}^{r_1} {m_2}^{r_2} \ldots {m_t}^{r_t}$ be the
factorization of $m$, where all $m_i , 1 \le i \le t$ , are prime and $m_i \ne m_j$ for $i\ne j$, then
the number of primitive words of length $m$ is equal to
$$
\begin{array}{rl}
|V|^m - & \sum\limits_{1\le i\le t}|V|^{\frac{m}{m_i}}
   + \sum\limits_{1\le i\le j \le t}|V|^{\frac{m}{m_i m_j}}\\\\
- & \sum\limits_{1\le i\le j \le k \le t}|V|^{\frac{m}{m_i m_j m_k}}
+
\cdots +(-1)^{t-1} |V|^{\frac{m}{m_1 m_2 \cdots m_t}}
\end{array}
$$
\end{proposition}

We observe that the deletion of a symbol from a $n$-length non-primitive word gives a maximum of $(n-1)$-different non-ins-robust primitive words  when the word is of type $a_1 a_2 \ldots a_{n}$ such that $a_i \ne a_{i+1}$ for $1 \le i \le n-1$ and minimum it can be $zero$ if the word is of type $a^r, \ r>2, \ a \in V$.
Given a word $w \in Z(n)$.
The number of words that can be obtained by deleting a symbol from $w$ is \\
$$0 \le |\{w_1 w_2 \mid w_1 a w_2 = w, ~w_1, w_2 \in V^*, a \in V \}| \le n.$$
We know from Lemma \ref{cor:prins} that a non-primitive word $w$ remains non-primitive after deleting a symbol $a$ if $w = a^n $ and $n \ge 3$. Otherwise a non-ins-robust primitive word.
$$Q_{\overline{I}}(n) = \{w_1 w_2 \mid w_1 a w_2 \in Z_{n+1}, a \in V, w_1, w_2 \in V^*\}$$

Therefore, the number of non-ins-robust primitive words of length $n$,
$Q_{\overline{I}}(n)$, is the difference between the number of all words obtained by deleting a symbol from
the words of set
$$Z(n+1) \setminus V^{n+1} \ where \ V^{n+1} =\{ a^{n+1} \mid a \in V \} \ for \ n \ge 2$$

We can find an upper bound on number of non-ins-robust primitive words of length
$n \ge 2$ as follows.
$$
\begin{array}{ll}
|Q_{\overline{I}}(n)|  = |\{w_1  w_2 \mid w =w_1 b w_2 \in Z(n+1)  \setminus V^{n+1} \\  ,w_1, w_2 \in V^*, b \in V \}|\\\\

|Q_{\overline{I}}(n)|  = |\{w_1  w_2 \mid w =w_1 a w_2 \in Z(n+1),\\  w_1, w_2 \in V^*, a \in V \}| - (n+1).|V|\\\\
 \le (n+1) . (|Z(n+1|) - |V|).
\end{array}
$$
From the Proposition \ref{count}, we know the number of primitive words of fixed
length. Thus the number of ins-robust-primitive words of length $n$, $Q_I(n)$, over an
alphabet $V$ is equal to $|Q_n| - |Q_{\overline{I}}(n)|$.

\section{Recognizing Ins-Robust Primitive Words} \label{sec:alg}
In this section, we give a linear time algorithm to determine if a given primitive word $w$ is ins-robust. We design the algorithm that exploits the property of the structure of ins-robust primitive words. We state some simple observations before presenting the algorithm.
The following theorem is based on the structure of ins-robust primitive word. We know that a word $w$ is ins-robust if it can not be written as $t^r t_1  t_2 t^s$ for $t \in Q$, $t = t_1 a t_2$ for some $a \in V$ , $r,s \ge 0$ and $r+s \ge 1$.

\begin{theorem} \label{insrp}
Let $u$ be a primitive word. Then $u$ will be \textbf{non-ins-robust} primitive word iff $uu$ contains at least one periodic word of length $|u|$ with period $p$ such that $p$ divides of length $|u| + 1$ and $p \le |u|$.
\end{theorem}

\begin{proof}
($\Rightarrow$) If $u$ is a non-ins-robust primitive word, then $u$ can be written as $t^r t_1  t_2 t^s$ for some primitive word $t$, $r+s \ge 1$ and $t = t_1 a t_2$ for some symbol $a \in V$ where $t_1 , t_2 \in V^*$.
$uu = t^r t_1 t_2 t^s t^r t_1 t_2 t^s$, This word contains a subword $ t_2 t^s t^r t_1$ of length $|u|$ that is $(t_2 t_1 a)^{r+s} t_2 t_1$ which is a periodic word with period $|t_2 t_1 a| = |t|$ which divides $|u|+1$.\\
($\Leftarrow$) Let $uu$ has a periodic substring of length $|u|$ with period p (divides $|u| +1$ and $p \le |u|$) where u is primitive word. Then $uu = t_1 x^r x' t_2$, where $t_1, t_2 \in V^*$, $|x^{r} x'|=|u|$, $x \in Q$, $r \ge 1$ and $x = x' a$ for some $a \in V$. $|t_1 t_2| = |u|$. Here we have two cases, either $x^r x'$ entirely contained in $u$ or some portion of $x^r x'$ contained in $u$.\\
\textbf{Case 1.} Let  $x^r x'$ entirely in $u$. Then $u$ is not ins-robust as $u = x^r x' $. \\
\textbf{Case 2.} Let some portion of $x^r x'$ contained in $u$. Since $uu = t_1 x^r x' t_2$, and $Z$ is reflective, therefore $t_2 t_1 x^r x' = u' u'$, where $u'$ is cyclic permutation of $u$. Hence, $u' = x^r x'$, is non-ins-robust. Since $Q_I$ is reflective, therefore $u$ is also non-ins-robust.

\end{proof}

\begin{corollary} \label{insr}
Let $u$ be a primitive word. Then $u$ will be \textbf{non-ins-robust} primitive word iff there exists a cyclic permutation of $u$, say $u'$, which is a periodic with period $p$ such that $p$ divides $|u| + 1$ and $p \le |u|$.
\end{corollary}

\begin{proof}
The proof follows from Theorem \ref{insrp}.

\end{proof}

Next we present a linear time algorithm to test ins-robustness of a primitive word by using the existing algorithm for finding maximal repetitions in linear time. For more details on the maximal repetition, see \cite{lothaire2005applied}.

\begin{algorithm}[!ht] 
\caption{\textsc{IsInsRobust}}\label{inropr}
\begin{algorithmic}[1]
\Require A finite word $u$
\Ensure ``True'' if $u$ is an ins-robust primitive word, else ``False''
\Statex
\State Let $v \leftarrow uu$.\label{step}
\State $ S \leftarrow $ \Call{FindMaximalRepetitions}{$v$}\label{maxrep1}
\Comment $S$ is a set of pairs of period and length.
\ForAll{$(p_i, l_i) \in S$} \label{st1}
  \If {$|u| \mod p_i = 0$ \textbf{and} $p_i < |u|$} \label{prim} \label{alg:prim} \Comment Testing primitivity.
    \State \Return False \Comment{The word $u$ is not primitive.}
  \EndIf

  \If {$p_i \le |u|$ \textbf{and} $(|u| + 1) \mod p_i = 0$ \textbf{and} $l_i \ge |u|$}\label{alg:main}
    \State \Return False \Comment{The word $u$ is not ins-robust. (Corollary \ref{insr})}
  \EndIf
\EndFor\label{st2}
\State \Return True \Comment{The word $u$ is ins-robust.}
\end{algorithmic}
\end{algorithm}

\begin{theorem}
Let $w$ be a word given as input to Algorithm \ref{inropr}. The algorithm returns true if and only if the word $w$ is ins-robust.
\end{theorem}
\begin{proof}
In step (\ref{maxrep1}), the algorithm finds the maximal repetitions with their periods.
Since $Q_{\overline{I}}$ is closed under reflectivity, therefore  $uu$ has all the cyclic permutations of $u$. There is a periodic word $x^r x'$, a permutation of $u$ such that $x=x' a $ for some $a\in V$. Therefore $uu$ also has this periodic word which is proved in Theorem \ref{insrp}. 
That is for a non-ins-robust primitive word $u$, $uu$ contains a periodic word of length at least $|u|$ with a period $p$ such that $p \ divides \ (|u|+1)$ and $p<|u|$. 
This is explained in step (\ref{alg:main}) where $u$ is a primitive word step (\ref{alg:prim}). 

\noindent  Otherwise $u$ is ins-robust primitive word. 
\end{proof}

\begin{theorem}
The property of being ins-robust primitive is testable on a word of length $n$ in $O(n)$ time.
\end{theorem}

\begin{proof} The Step (\ref{step}) in Algorithm \ref{inropr} running time is an O(1) operation. In step (\ref{maxrep1}) maximal repetition algorithm given in \cite{lothaire2005applied} is used which has linear time complexity. Now from step (\ref{st1}) to step (\ref{st2}), the complexity depends on the cardinality of $S$, which is less than $n$. Hence it also has linear time complexity. Therefore by theorem (\ref{insrp}) testing ins-robustness for primitive word is linear.
\end{proof}

\section{Conclusions and Future Work} \label{sec:con}
We have discussed a special subclass of the language of primitive words which is known as
ins-robust primitive words. We have characterized ins-robust
primitive words and identified several properties and proved that the language
of ins-robust primitive words $Q_I$ is not regular. We also proved that the language of non-ins-robust primitive words $Q_{\overline{I}}$ is not context-free. We identified that $Q_I$ is dense over an alphabet $V$. We
have also presented a linear time algorithm to test if a given word is ins-robust primitive.
Finally, we have given a lower bound on the number of ins-robust primitive words of a given length.

There are several interesting questions that remain unanswered about
ins-robust primitive words. Some of them that we plan to explore in
immediate future are as follows. Is $Q_I ^{i}$ for $i \ge 2$ regular? It is known that $Q^i$ for $i \ge 2$ is regular \cite{doi}. We also
conjecture that the language of ins-robust primitive words $Q_I$ is not a
deterministic context-free language. We believe that the properties we have
identified for ins-robust primitive words will be helpful in answering
these questions.


\bibliography{am_ref}{}
\bibliographystyle{plain}



\end{document}